\documentclass[11pt]{article}

\usepackage{fullpage}
\usepackage{amsmath}
\usepackage{amstext}
\usepackage{amssymb}
\usepackage{amsthm}
\usepackage{amsfonts}
\usepackage{amsxtra}
\usepackage{mathrsfs}
\usepackage{mathtools}
\usepackage{dsfont}
\usepackage{breqn}
\usepackage{enumerate}
\usepackage{bbm}
\usepackage{comment}
\usepackage[utf8]{inputenc}

\usepackage{sty}

\usepackage{tikz}
\usetikzlibrary{patterns}
\usetikzlibrary{hobby} 
\usetikzlibrary{arrows.meta} 
\tikzset{>={Latex[length=4,width=4]}} 
\usetikzlibrary{calc,intersections,decorations.markings}
\usepackage{siunitx}

\colorlet{mylightblue}{blue!20}
\colorlet{myblue}{blue!50!black}
\colorlet{mydarkblue}{blue!30!black}
\colorlet{mylightred}{red!10}
\colorlet{myred}{red!50!black}
\colorlet{mydarkred}{red!60!black}
\colorlet{mydarkgreen}{green!30!black}

\tikzset{
midarr/.style={decoration={markings,mark=at position #1 with {\arrow{stealth}}},postaction={decorate}},
midarr/.default=0.5
}

\usepackage{wrapfig}

\newtheorem{theorem}{Theorem}[section]
\newtheorem{lemma}[theorem]{Lemma}

\newtheorem{corollary}[theorem]{Corollary}

\newtheorem{definition}[theorem]{Definition}

\usepackage{hyperref}
\usepackage{xcolor}
\definecolor{OliveGreen}{HTML}{3C8031}
\hypersetup{
colorlinks=true,
urlcolor=blue,
linkcolor=blue,
citecolor=OliveGreen,
}

\title{Impossibility of latent inner product recovery via rate distortion}

\usepackage{authblk}

\author{Cheng Mao}
\author{Shenduo Zhang}
\affil{School of Mathematics, Georgia Institute of Technology}

\begin{document}

\maketitle

\begin{abstract}
In this largely expository note, we present an impossibility result for inner product recovery in a random geometric graph or latent space model using the rate-distortion theory. More precisely, suppose that we observe a graph $A$ on $n$ vertices with average edge density $p$ generated from Gaussian or spherical latent locations $z_1, \dots, z_n \in \R^d$ associated with the $n$ vertices. It is of interest to estimate the inner products $\langle z_i, z_j \rangle$ which represent the geometry of the latent points. We prove that it is impossible to recover the inner products if $d \gtrsim n h(p)$ where $h(p)$ is the binary entropy function. This matches the condition required for positive results on inner product recovery in the literature. The proof follows the well-established rate-distortion theory with the main technical ingredient being a lower bound on the rate-distortion function of the Wishart distribution which is interesting in its own right.
\end{abstract}




\section{Introduction}

Random graphs with latent geometric structures comprise an important class of network models used across a broad range of fields~\cite{penrose2003random,hoff2002latent,barthelemy2011spatial}. 
In a typical formulation of such a model, each vertex of a graph on $n$ vertices is assumed to be associated with a latent location $z_i \in \R^d$ where $i=1,\dots,n$. 
With $A \in \{0,1\}^{n \times n}$ denoting the adjacency matrix of the graph, each edge $A_{ij}$ follows the Bernoulli distribution with probability parameter $\kappa(z_i,z_j)$, where $\kappa : \R^d \times \R^d \to [0,1]$ is a kernel function.
In other words,
the edges of the graph are formed according to the geometric locations of the vertices in a latent space.
Given the graph $A$, the central question is then to recover the latent geometry, formulated as estimating the inner products $\langle z_i, z_j \rangle$\footnote{One can also formulate the problem as as estimating the pairwise distances $\{\|z_i - z_j\|_2\}_{i,j=1}^n$ which is essentially equivalent to inner product estimation. The problem is not formulated as estimating the latent locations $\{z_i\}_{i=1}^n$ themselves, because the kernel function $\kappa$ is typically invariant under an orthogonal transformation of $z_1, \dots, z_n$, making them non-identifiable.}. 


In the study of this class of random graphs, a Gaussian or spherical prior is often imposed on the latent locations $z_1, \dots, z_n$, including in the early work on latent space models \cite{hoff2002latent,handcock2007model,hoff2007modeling,krivitsky2009representing} and in the more recent work on random geometric graphs \cite{araya2019latent,eldan2022community,li2023spectral}. 
In particular, the isotropic spherical or Gaussian prior allows the latter line of work to use the theory of spherical harmonics to analyze spectral methods for estimating the latent inner products.
For a class of kernels including the step function $\kappa(z_i, z_j) = \bbone\{\langle z_i, z_j \rangle \ge \tau\}$ for a threshold $\tau$, it is known (see Theorem~1.4 of \cite{li2023spectral}) that the inner products can be estimated consistently if $d \ll n h(p)$ where $p$ is the average edge density of the graph and $h(p)$ is the binary entropy function.
However, a matching negative result was not established (as remarked in Section~1.3 of \cite{li2023spectral}).

In this largely expository note, we close this gap by proving in Corollary~\ref{cor:rgg} that it is information-theoretically impossible to recover the inner products in a random geometric graph model if $d \gsm nh(p)$, thereby showing that $d \asymp n h(p)$ is indeed the recovery threshold\footnote{Another related statistical problem is testing a random geometric graph model against an Erd\H{o}s--R\'enyi graph model with the same average edge density \cite{bubeck2016testing}. This testing threshold, or detection threshold, is conjectured to be $d \asymp (n h(p))^3$, and the lower bound is still largely open. See \cite{bubeck2016testing,brennan2020phase,liu2022testing}.}. 
In fact, it is not difficult to predict this negative result from entropy counting: It is impossible to recover the geometry of $n$ vectors in dimension $d$ from $\binom{n}{2}$ binary observations with average bias $p$ if $nd \gtrsim \binom{n}{2} h(p)$ since there is not sufficient entropy.
And this argument does not rely on the specific model (such as the kernel function $\kappa$) for generating the random graph $A$.

To formalize the entropy counting argument, the rate-distortion theory \cite{shannon1959coding} provides a standard approach (see also \cite{cover1999elements,YihongITbook} for a modern introduction).
The key step in this approach is a lower bound on the rate-distortion function of the estimand, i.e., $X \in \R^{n \times n}$ with $X_{ij} := \langle z_i, z_j \rangle$ in our case. 
If $z_1, \dots, z_n$ are isotropic Gaussian vectors, then $X$ follows the Wishart distribution. 
Therefore, our main technical work lies in estimating the rate-distortion function for the Wishart distribution (and its variant when $z_1, \dots, z_n$ are on a sphere), which has not been done explicitly in the literature to the best of our knowledge. See Theorem~\ref{thm:wishart-rate-distortion}.  

The technical problem in this note is closely related to a work \cite{lee2017near} on low-rank matrix estimation. To be more precise, 
Theorem~VIII.17 of \cite{lee2017near} proves a lower bound on the rate-distortion function of a rank-$d$ matrix $X = Z Z^\top$ where $Z \in \R^{n \times d}$.
Our proof partly follows the proof of this result but differs from it in two ways: First, the result of \cite{lee2017near} assumes that $Z$ is uniformly distributed on the Stiefel manifold, i.e., the columns of $Z$ are orthonormal, while we assume that $Z$ has i.i.d.\ Gaussian or spherical rows. Without the simplification from the orthonormality assumption, our proof requires different linear algebraic technicalities. Second, the result of \cite{lee2017near} focuses on $d \le n$, while we also consider the case $d > n$ which requires a completely different proof.

Finally, as a byproduct of the lower bound on the rate-distortion function of $X$, we present in Corollary~\ref{cor:one-bit} an impossibility result for one-bit matrix completion. 
While one-bit matrix completion has been studied extensively in the literature \cite{davenport20141,JMLR:v14:cai13b,bhaskar20151}, less is known for the Bayesian model where a prior is assumed on the matrix $X$ to be estimated \cite{cottet20181,mai2024concentration}. 
Similar to inner product estimation from a random geometric graph, the goal of one-bit matrix completion is to estimate a (typically low-rank) matrix $X$ from a set of binary observations. 
It is therefore plausible that many techniques for random graphs can be used for one-bit matrix completion, and vice versa.
This note provides such an example.

\section{Main results}
\label{sec:proof of main}

In this section, we study the rate-distortion function for the Wishart distribution and its spherical variant.
Let $I(X;Y)$ denote the mutual information between random variables $X$ and $Y$.
The rate-distortion function is defined as follows (see Part~V of \cite{YihongITbook}).






\begin{definition}[Rate-distortion function]\label{def:rate-distortion function}
Let $X$ be a random variable taking values in $\R^\ell$, and let $P_{Y \mid X}$ be a conditional distribution on $\R^\ell$ given $X$.
Let $L$ be a distortion measure (or a loss function), i.e., a bivariate function $L : \R^\ell \times \R^\ell \to \R_{\ge 0}$. 
For $D>0$, the rate-distortion function of $X$ with respect to $L$ is defined as
\begin{equation*}
R_X^L(D) := \inf_{P_{Y \mid X} : \E L(X,Y) \le D} I(X;Y).
\end{equation*}
\end{definition}

The main technical result of this note is the following lower bound on the rate-distortion function of a Wishart matrix. 

\begin{theorem}[Rate-distortion function of a Wishart matrix]
\label{thm:wishart-rate-distortion}
For positive integers $n$ and $d$, let $Z := [z_1 \dots z_n]^\top \in \R^{n \times d}$ where the i.i.d.\ rows $z_1, \dots, z_n$ follow either the Gaussian distribution $\cN(0, \frac 1d I_d)$ or the uniform distribution on the unit sphere $\cS^{d-1} \subset \R^d$. 
Let $X := Z Z^\top$. 
Define a loss function\footnote{The normalization in the definition of $L$ is chosen so that the trivial estimator $\E X = I_n$ of $X$ has risk $\E L(X,\hat X) = 1$ in the case of Gaussian $z_i$, since 
$\E [X_{ij}^2] = \E [\ar{z_i,z_j}^2] = 1/d$ 
for $i\neq j$ and 
$\E[(X_{ii} - 1)^2]  = \E [(\ar{z_i,z_i} - 1)^2] = 2/d$.}
\begin{equation}
L(X,\hat X) := \frac{d}{n(n+1)}\|X-\hat X\|_F^2.
\label{eq:loss-function-L}
\end{equation}
Let $n \land d := \min\{n,d\}$.
There is an absolute constant $c>0$ such that for any $D \in (0, c)$, we have
$$
R_X^L(D) \ge c n (n \land d) \log \frac{1}{D} .
$$
\end{theorem}

For $d < n$, the $n \times n$ matrix $X$ is rank-deficient and is a function of $Z \in \R^{n \times d}$, so we expect the order $nd$ for the rate-distortion function; for $d \ge n$, we expect the order $n^2$ considering the size of $X$.
The matching upper bound on the rate-distortion function can be obtained using a similar argument as that in Section~\ref{sec:small d} for small $d$ and through a comparison with the Gaussian distribution for large $d$ (see Theorem~26.3 of \cite{YihongITbook}). 
Since it is in principle easier to obtain the upper bound and only the lower bound will be used in the downstream statistical applications, we do not state it here. 
Moreover, at the end of this section, we discuss the best possible constant $c$ in the above lower bound. 
The bulk of the paper, Section~\ref{sec:pf-rate-dist}, will be devoted to proving Theorem~\ref{thm:wishart-rate-distortion}.
With this theorem in hand, we first establish corollaries for two statistical models via entropy counting. 

\begin{corollary}[Random geometric graph or latent space model]
\label{cor:rgg}
Fix positive integers $n, d$ and a parameter $p \in (0,1)$.
Suppose that we observe a random graph on $n$ vertices with adjacency matrix $A$ 
with average edge density $p$, i.e., 
$\sum_{(i,j) \in \binom{[n]}{2}} \E[A_{ij}] = \binom{n}{2} p.$
Suppose that $A$ is generated according to an arbitrary model from the latent vectors $z_1, \dots, z_n$ given in Theorem~\ref{thm:wishart-rate-distortion}, and the goal is to estimate the inner products $X_{ij} := \langle z_i, z_j \rangle$ in the norm $L$ defined in \eqref{eq:loss-function-L}. 
If $d \ge c n h(p)$ where $c > 0$ is any absolute constant and $h(p) := -p \log p - (1-p) \log (1-p)$ is the binary entropy function, then for any estimator $\hat X$ measurable with respect to $A$, we have $\E L(X, \hat X) \ge D$ for a constant $D = D(c) >0$.
\end{corollary}

\begin{proof}
The estimand $X$, the observation $A$, and the estimator $\hat X$ form a Markov chain $X \to A \to \hat X$.
By the data processing inequality, we have 
$$
I(X; \hat X) \leq I(A; \hat X) \leq H(A) ,
$$
where $H(A)$ denotes the entropy of $A$.
Since $\sum_{(i,j) \in \binom{[n]}{2}} \E[A_{ij}] = \binom{n}{2} p$, by the maximum entropy under the Hamming weight constraint (see Exercise~I.7 of \cite{YihongITbook}), we get
$$
H(A) \le \binom{n}{2} h(p) .
$$
If $\E L(X,\hat X) \le D$, then combining the above inequalities with Theorem~\ref{thm:wishart-rate-distortion} gives
$$
c n(n \land d) \log \frac 1D \le R_X^L(D) \le I(X, \hat X) \le \binom{n}{2} h(p) .
$$
Taking $D > 0$ to be a sufficiently small constant, we then get $n \land d < c n h(p)$, i.e., $d < c n h(p)$. 
\end{proof}

As a second application of Theorem~\ref{thm:wishart-rate-distortion}, we consider one-bit matrix completion with a Wishart prior.

\begin{corollary}[One-bit matrix completion]
\label{cor:one-bit}
Fix positive integers $n, d$ and a parameter $p \in (0,1)$.
Suppose that $X \in \R^{n \times n}$ is a rank-$d$ matrix to be estimated. 
Assume the prior distribution of $X$ as given in Theorem~\ref{thm:wishart-rate-distortion}. 
For each entry $(i,j) \in [n]^2$, suppose that with probability $p_{ij}$, we have a one-bit observation $A_{ij} \in \{0,1\}$ according to an arbitrary model, and with probability $1-p_{ij}$, we do not have an observation, denoted as $A_{ij} = \ast$.
Let $p$ be the average probability of observations, i.e., $\sum_{i,j=1}^n p_{ij} = n^2 p$. 
Let $L$ be the loss function defined in \eqref{eq:loss-function-L}. 
If $d \ge c n (h(p) + p)$ where $c>0$ is any absolute constant and $h(p) := -p \log p - (1-p) \log (1-p)$, then for any estimator $\hat X$ measurable with respect to $A$, we have $\E L(X, \hat X) \ge D$ for a constant $D = D(c) > 0$.
\end{corollary}

\begin{proof}
The argument is the same as the proof of Corollary~\ref{cor:rgg}, except the bound on the entropy of $A$. 
Let $Z \in \{0,1\}^{n \times n}$ have Bernoulli$(p_{ij})$ entries such that $Z_{ij} = \bbone\{A_{ij} \ne \ast\}$. 
Then we have the conditional entropy $H(Z \mid A) = 0$.
Conditional on any value of $Z$, the entropy of $A$ is at most $\log 2^{\|Z\|_1}$.
As a result,  
$$
H(A \mid Z) 
\le \E_Z \log 2^{\|Z\|_1} = n^2 p \log 2 .
$$
We therefore obtain 
$$
H(A) = H(A \mid Z) + I(Z; A) = H(A \mid Z) + H(Z) \le n^2 (h(p) + p \log 2). 
$$
The rest of the proof is the same as that for the random geometric graph model.
\end{proof}

\paragraph{Open problems.}
Several interesting problems are left open.
\begin{itemize}
\item \emph{Sharp constant:}
Recall that the lower bound on the rate-distortion function of the Wishart distribution in Theorem~\ref{thm:wishart-rate-distortion}. 
While the order $n (n \land d) \log \frac 1D$ is believed to be optimal, we did not attempt to obtain the sharp constant factor. 
In the case $d \ge n$, the rate-distortion function can be bounded from above by that of a Gaussian Wigner matrix, and the best leading constant is $1/4$ (see Theorems~26.2 and~26.3 of \cite{YihongITbook}). 
Indeed, the end result of Section~\ref{sec:large d} indeed shows a lower bound with the constant $1/4$ in the leading term if $D \to 0$. In the case $d/n \to 0$, Lemma~\ref{lem:rate-distortion-rotation moduling rotation} suggests that the best constant may be $1/2$, but we did not make the effort to obtain it as the end result.
The most difficult situation appears to be when $d < n = O(d)$, in which case our techniques fail to obtain any meaningful constant factor.

\item \emph{Optimal rate:} Combined with the work \cite{li2023spectral}, Corollary~\ref{cor:rgg} gives the recovery threshold $d \asymp n h(p)$ for random geometric graphs with Gaussian or spherical latent locations. 
However, it remains open to obtain an optimal lower bound on $\E L(X,\hat X)$ as a function of $d,n,p$ in the regime $d \ll n h(p)$. 
We believe the simple approach of entropy counting is not sufficient for obtaining the optimal rate and new tools need to be developed.

\item \emph{General latent distribution:} 
Existing positive and negative results for estimation in random geometric graph models are mostly limited to isotropic distributions of latent locations, such as Gaussian or spherical in \cite{araya2019latent,eldan2022community,li2023spectral} and this work. It is interesting to extend these results to more general distributions and metric spaces; see \cite{bangachev2023detection,bangachevRandomAlgebraicGraphs2023a} for recent work. Even for random geometric graphs with anisotropic Gaussian latent points, while there has been progress on the detection problem \cite{eldan2020information,brennan2024threshold}, extending the recovery results to the anisotropic case remains largely open.
\end{itemize}

\section{Proof of Theorem~\ref{thm:wishart-rate-distortion}}
\label{sec:pf-rate-dist}

Let $\ConstQFEJAIKNU \in (0,1)$ be some absolute constant to be determined later. 
We first consider the Gaussian model where $z_i \sim \cN(0, \frac 1d I_d)$. The proof is split into three cases $d \le \ConstQFEJAIKNU n$, $d \geq n$, and $\ConstQFEJAIKNU n < d < n$, proved in Sections~\ref{sec:small d}, \ref{sec:large d}, and~\ref{sec:middle d} respectively.
We then consider the spherical model in Section~\ref{sec:spherical}.




\subsection{Case $d \le \ConstQFEJAIKNU n$}
\label{sec:small d}

To study the rate-distortion function of $X = ZZ^\top$, we connect it to the rate-distortion function of $Z$ in the distortion measure to be defined in \eqref{eq:loss-function-small-l}.
The strategy is inspired by \cite{lee2017near}, but the key lemma connecting the distortion of $X$ to that of $Z$ is different. 
For $Z, \hat Z \in \R^{d \times d}$, define a loss function for recovering $Z$ up to an orthogonal transformation
\begin{equation}
\ell(Z,\hat Z) := \frac{1}{n} \inf_{O\in\cO(d)} \|Z-\hat Z O \|_F^2 ,
\label{eq:loss-function-small-l}
\end{equation}
where $\cO(d)$ denotes the orthogonal group in dimension $d$. 
The normalization is chosen so that $\E \ell(Z,\E Z) = \E \ell(Z,0) = 1$. 
We start with a basic linear algebra result.

\begin{lemma}
\label{lem:lin-alg-ineq}
Let $A, B \in \R^{n \times d}$. For the loss functions $L$ and $\ell$ defined by \eqref{eq:loss-function-L} and \eqref{eq:loss-function-small-l} respectively, we have
$$
\ell(A,B) \le \sqrt{\frac{n+1}{n} L(AA^\top, BB^\top)} .
$$
\end{lemma}

\begin{proof}
Consider the polar decompositions
$A = (A A^\top)^{1/2} U$ and $B = (B B^\top)^{1/2} V$ where $U, V \in \cO(d)$. 
Then we have
\begin{align*}
\ell(A, B) &= \frac{1}{n} \inf_{O\in\cO(d)} \|A - B O \|_F^2 \\
&\le \frac{1}{n} \|(A A^\top)^{1/2} U - (B B^\top)^{1/2} V (V^\top U) \|_F^2 \\
&= \frac{1}{n} \|(A A^\top)^{1/2} - (B B^\top)^{1/2} \|_F^2 .
\end{align*}
The Powers--St{\o}rmer inequality \cite{powers1970free} gives
$$
\|(A A^\top)^{1/2} - (B B^\top)^{1/2} \|_F^2
\le \|A A^\top - B B^\top\|_* ,
$$
where $\|\cdot\|_*$ denotes the nuclear norm.
In addition, $A A^\top$ and $B B^\top$ are at most rank $d$, so 
$$
\ell(A,B) \le \frac 1n \|A A^\top - B B^\top\|_*
\le \frac{\sqrt{d}}{n} \|A A^\top - B B^\top\|_F = \sqrt{\frac{n+1}{n} L(AA^\top, BB^\top)} .
$$
\end{proof}

Next, we relate the rate-distortion function of $X = ZZ^\top$ in the loss $L$ to the rate-distortion function of $Z$ in the loss $\ell$.

\begin{lemma}
\label{lem:rate-distortion-comparison}
Let $Z$ and $X$ be defined as in Theorem~\ref{thm:wishart-rate-distortion}, and let $L$ and $\ell$ be defined by \eqref{eq:loss-function-L} and \eqref{eq:loss-function-small-l} respectively. 
Recall the notation of the rate-distortion function in Definition~\ref{def:rate-distortion function}. For $D > 0$, we have
$$
R_X^L(D) \ge R_Z^\ell(\sqrt{8D}) .
$$
\end{lemma}

\begin{proof}
Fix a conditional distribution $P_{Y \mid X}$ such that $\E L(X,Y) \le D$. Define 
$$
\tilde Z = \argmin_{W \in \R^{n \times d}} \|Y - W W^\top \|_F ,
$$
where the non-unique minimizer $\tilde Z$ is chosen arbitrarily. 
Then we have
$$
\|Z Z^\top - \tilde Z \tilde Z^\top\|_F
\le \|Z Z^\top - Y\|_F + \|Y - \tilde Z \tilde Z\|_F
\le 2 \|Z Z^\top - Y\|_F .
$$
In other words,
$$
L(Z Z^\top, \tilde Z \tilde Z^\top) \le 4 L(X,Y) .
$$
By Lemma~\ref{lem:lin-alg-ineq}, 
$$
\ell(Z, \tilde Z) \le \sqrt{2 L(Z Z^\top, \tilde Z \tilde Z^\top)} 
\le \sqrt{8 L(X,Y)} .
$$
Jensen's inequality then yields
$$
\E \ell(Z, \tilde Z) \le \E \sqrt{8 L(X,Y)} \le \sqrt{8 \E L(X,Y)} \le \sqrt{8D} .
$$


Let $O$ be a uniform random orthogonal matrix over $\cO(d)$, independent from everything else. In view of the definition of $\ell$, we have 
$$
\E \ell(Z O, \tilde Z) = \E \ell(Z, \tilde Z) \le \sqrt{8D} .
$$
Therefore, by the definition of the rate-distortion function $R_Z^\ell$ (see Definition~\ref{def:rate-distortion function}), 
$$
I(Z O; \tilde Z) \ge R_{ZO}^\ell(\sqrt{8D}) 
= R_Z^\ell(\sqrt{8D}) ,
$$
where the equality follows from the orthogonal invariance of the distribution of $Z$.

Next, we note that
$$
I(Z O; \tilde Z) \le I(Z Z^\top; \tilde Z) .
$$
(In fact, equality holds because the reverse inequality is trivial by data processing.) 
To see this, given $Z Z^\top$, take any $A \in \R^{n \times d}$ such that $Z Z^\top = A A^\top$, and let $Q$ be a uniform random orthogonal matrix over $\cO(d)$ independent from everything else. 
Since $A = Z P$ for some $P \in \cO(d)$, we have $(A Q, \tilde Z) = (Z P Q, \tilde Z) \stackrel{d}{=} (Z O, \tilde Z)$, where $\stackrel{d}{=}$ denotes equality in distribution. 
Hence, the data processing inequality gives $I(Z Z^\top; \tilde Z) \ge I(AQ; \tilde Z) = I(ZO; \tilde Z)$.


Combining the above two displays and recalling that $\tilde Z$ is defined from $Y$, we apply the data processing inequality again to obtain
$$
I(X; Y) \ge I(ZZ^\top; \tilde Z \tilde Z^\top) \ge R_Z^\ell(\sqrt{8D}) .
$$
Minimizing $P_{Y \mid X}$ subject to the constrain $\E L(X,Y) \le D$ yields the the rate-distortion function $R_X^L(D)$ on the left-hand side, completing the proof.
\end{proof}

\begin{lemma}
\label{lem:rate-distortion-rotation moduling rotation}
Let $Z$ be defined as in Theorem~\ref{thm:wishart-rate-distortion}, let $\ell$ be defined by \eqref{eq:loss-function-small-l}, and let $R_Z^\ell$ be given by Definition~\ref{def:rate-distortion function}. There is an absolute constant $C>0$ such that for any $D \in (0,1/4)$, we have
$$
R_Z^\ell(D) \ge \frac{nd}{2} \log \frac{1}{4D} - \frac{d^2}{2} \log \frac{C}{D} .
$$
\end{lemma}

%

\begin{proof}
Fix 
a conditional distribution $P_{\hat Z \mid Z}$ 
such that $\E \ell(Z,\hat Z) \leq D$. 
Let $O = O(Z,\hat Z) \in \cO(d)$ be such that $\frac{1}{n} \|\hat Z O - Z\|_F^2 = \ell(Z,\hat Z)$.
Then we have $\E \|\hat Z O - Z\|_F^2 \le nD$.
Let $N(\cO(d),\epsilon)$ be an $\epsilon$-net of $\cO(d)$ with respect to the Frobenius norm, where $\epsilon^2 = \frac{nD}{\E \norm{Z}_2^2} \wedge d$. For $O = O(Z, \hat Z)$, choose $\hat O = \hat O(Z, \hat Z) \in N(\cO(d),\epsilon)$ such that $\|\hat O - O\|_F^2 \leq \epsilon^2$. Define $W := \hat Z \hat O$. We have
\begin{equation*}
\begin{aligned}
\E \|W - Z\|_F^2 & = \E \| \hat Z \hat O - Z \|_F^2 = \E \| \hat Z - Z\hat O^{-1} \|_F^2                                                           \\
& \leq 2\E \| \hat Z - Z O^{-1} \|_F^2 + 2\E \| Z O^{-1} - Z\hat O^{-1} \|_F^2                                                    \\
& \leq 2\E \|\hat Z O - Z\|_F^2 + 2 \E \|Z\|^2\|O^{-1}-\hat O^{-1}\|_F^2                                                        \\
&\le 2nD + 2 \epsilon^2 \E\|Z\|^2 = 4nD ,
\end{aligned}
\end{equation*}
where $\|\cdot\|$ denotes the spectral norm.

By Theorem 26.2 of \cite{YihongITbook} (with $d$ replaced by $nd$ and $\sigma^2$ replaced by $1/d$), the rate-distortion function of $Z$ with respect to the Frobenius norm $L_0(Z,W) := \|Z-W\|_F^2$ is
\begin{equation}
R_Z^{L_0}(D) = \frac{nd}{2} \log \frac{n}{D} .
\label{eq:gaussian-matrix-rate-distortion}
\end{equation}
Since $\E\|W-Z\|_F^2 \le 4nD$, we obtain
$$
I(Z;W) \ge R_Z^{L_0}(4nD) = \frac{nd}{2} \log \frac{1}{4D} .
$$
Moreover, we have
$$
I(Z;W) \le I(Z;\hat Z,\hat O) 
= I(Z;\hat Z) + I(Z;\hat O \mid \hat Z) 
\leq I(Z;\hat Z) + H(\hat O) ,
$$
where the three steps follow respectively from the data processing inequality, the definition of conditional mutual information $I(Z;\hat O \mid \hat Z)$, and a simple bound on the mutual information by the entropy.
The above two inequalities combined imply
$$
I(Z; \hat Z) \ge I(Z; W) - H(\hat O) \ge \frac{nd}{2} \log \frac{1}{4D} - H(\hat O) .
$$

Since $\hat O \in N(\cO(d),\eps)$, the entropy $H(\hat O)$ can be bounded by the metric entropy of $\cO(\eps)$.
By Theorem 8 of \cite{szarekMetricEntropyHomogeneous1997}, there is an absolute constant $C_0>1$ such that the covering number of $\cO(d)$ with respect to the Frobenius norm is at most $\pr{\frac{\sqrt{C_0 d}}{\epsilon}}^{d^2}$ 
for any 
$\epsilon\in (0,\sqrt{d})$. 
We have
$$ 
\eps = \sqrt{\frac{nD}{\E \norm{Z}^2}} \land \sqrt{d} \ge c_1 \sqrt{dD} 
$$
for an absolute constant $c_1 > 0$, where the bound follows from the concentration of $\|Z\|$ at order $O(\frac{\sqrt{n} + \sqrt{d}}{\sqrt{d}})$ (see, e.g., Corollary~5.35 of \cite{vershyninIntroductionNonasymptoticAnalysis2010}) and that $d \le n$.
Therefore,
$$
H(\hat O) \le \log |N(\cO(d)|
\le \frac{d^2}{2} \log \frac{C_0 d}{\eps^2}
\le \frac{d^2}{2} \log \frac{C_0}{c_1^2 D} .
$$

Putting it together, we obtain
$$
I(Z; \hat Z) 
\ge \frac{nd}{2} \log \frac{1}{4D} - \frac{d^2}{2} \log \frac{C_0}{c_1^2 D} ,
$$
finishing the proof
in view of the definition of $R_Z^L(D)$.
\end{proof}

Combining Lemmas~\ref{lem:rate-distortion-comparison} and~\ref{lem:rate-distortion-rotation moduling rotation}, we conclude that
$$
R_X^L(D) \ge \frac{nd}{2} \log \frac{1}{4\sqrt{8D}} - \frac{d^2}{2} \log \frac{C}{\sqrt{8D}} 
\ge \frac{nd}{8} \log \frac{1}{D}
$$
provided that $D \in (0,c^*)$ and $d \le c^* n$ for a sufficiently small constant $c^* > 0$.

\subsection{Case $d \geq  n$}
\label{sec:large d}

In the case $d \ge n$, the Wishart distribution of $X = ZZ^\top$ has a density on the set of symmetric matrices $\R^{n(n+1)/2}$, and we can apply the Shannon lower bound \cite{shannon1959coding} on the rate-distortion function. See Equation~(26.5) and Exercise~V.22 of the book \cite{YihongITbook} (with the norm taken to be the Euclidean norm and $r=2$) for the following result.

\begin{lemma}[Shannon lower bound \cite{shannon1959coding}]
Let $Y$ be a continuous random vector with a density on $\R^N$.
For $D>0$, let $R_Y^{L_0}(D)$ be the rate-distortion function of $Y$ with respect to the Euclidean norm $L_0(Y, \hat Y) := \|Y - \hat Y\|_2^2$. 
Let $\ch(Y)$ denote the differential entropy of $Y$.
Then we have
\begin{equation*}
R_Y^{L_0}(D) \ge \ch(Y) - \frac{N}{2} \log \frac{2 \pi e D}{N} .
\end{equation*}
\end{lemma}

\noindent
As a result, for the loss $L$ defined by \eqref{eq:loss-function-L} and the random matrix $X$ distributed over $\R^{n(n+1)/2}$, we have
$$
R_X^L(D) \ge \ch(X) - \frac{n(n+1)}{4} \log \frac{4 \pi e D}{d} .
$$
The differential entropy $\ch(X)$ of the Wishart matrix $X$ is known.

\begin{lemma}[Differential entropy of a Wishart matrix \cite{lazoEntropyContinuousProbability1978}]
For $X$ defined in Theorem~\ref{thm:wishart-rate-distortion}, we have
\begin{equation*}
\ch(X) = \frac{n(n+1)}{2} \log \frac{2}{d} + \log \Gamma_n\pr{\frac{d}{2}}-\frac{d-n-1}{2} \psi_n \pr{\frac{d}{2}}+\frac{nd}{2},
\end{equation*}
where $\Gamma_n$ is the multivariate gamma function and $\psi_n$ is the multivariate digamma function.
\end{lemma}



\noindent
The above two results combined give the lower bound
\begin{align}
R_X^L(D) &\ge \frac{n(n+1)}{2} \log \frac{2}{d} + \log \Gamma_n\pr{\frac{d}{2}}-\frac{d-n-1}{2} \psi_n \pr{\frac{d}{2}}+\frac{nd}{2} - \frac{n(n+1)}{4} \log \frac{4 \pi e D}{d} \notag \\
&= \frac{nd}{2} + \frac{n(n+1)}{4} \log \frac{1}{\pi e D d} + \log \Gamma_n\pr{\frac{d}{2}}-\frac{d-n-1}{2} \psi_n \pr{\frac{d}{2}} . \label{eq:slb-consequence}
\end{align}


We now analyze the functions $\Gamma_n$ and $\gamma_n$.
By Stirling's approximation for the gamma function (see Equation~6.1.40 of \cite{abramowitz1948handbook}), we have $\log \Gamma(x+1/2) \geq x \log (x+1/2) - x - 1/2 + \frac{\log(2\pi)}{2}$ for $x \ge 0$. 
Together with the definition of the multivariate gamma function $\Gamma_n$, this gives
\begin{align*}
\log \Gamma_n\pr{\frac{d}{2}} & = \frac{n(n-1)}{4}\log \pi + \sum_{i=1}^n \log \Gamma\pr{\frac{d+1-i}{2}}                                                                                    \\
& \geq \frac{n(n-1)}{4}\log \pi + \sum_{i=1}^n \left( \frac{d-i}{2} \log \frac{d+1-i}{2} - \frac{d+1-i}{2} + \frac{\log(2\pi/e)}{2} \right)                                    \\
&\ge \frac{n^2}{4}\log(\pi e) - \frac{nd}{2} + \sum_{i=1}^n \left( \frac{d-i}{2} \log \frac{d+1-i}{2} \right)  - O(n) .
\end{align*}
Moreover, by Equation (2.2) of \cite{alzer1997some}, the digamma function satisfies $\log x - \frac 1x < \psi(x) < \log x$ for $x>0$. 
Combining this with the definition of the multivariate digamma function $\psi_n$, we obtain
\begin{align*}
\frac{d-n-1}{2} \psi_n \pr{\frac{d}{2}}
&= \frac{d-n-1}{2} \sum_{i=1}^n \psi\left( \frac{d+1-i}{2} \right) \\
&\le 
\frac{d-n-1}{2} \sum_{i=1}^n \log \frac{d+1-i}{2} + O(n) ,
\end{align*}
where we note that the $O(n)$ term is only necessary in the case that $d=n$ and $\frac{d-n-1}{2}$ is negative.


Plugging the above two estimates into \eqref{eq:slb-consequence}, we see that
\begin{equation}
\label{eq:rxl-lower-step}
R_X^L(D) 
\ge \frac{n(n+1)}{4} \log \frac{1}{D d} + \sum_{i=1}^n \left( \frac{n+1-i}{2} \log \frac{d+1-i}{2} \right) - O(n) .
\end{equation}
If $d \ge 2n$, then 
\begin{align*}
R_X^L(D) 
&\ge \frac{n(n+1)}{4} \log \frac{1}{D d} + \left( \log \frac{d+1-n}{2} \right) \sum_{i=1}^n \frac{n+1-i}{2} - O(n) \\
&= \frac{n(n+1)}{4} \log \frac{1}{D} + \frac{n(n+1)}{4} \log \frac{d+1-n}{2d} - O(n) \\
&\ge \frac{n(n+1)}{4} \log \frac{1}{D} - O(n^2) .
\end{align*}
For $n \le d < 2n$, we first note that the term $\frac{n+1-i}{2} \log \frac{d+1-i}{2}$ with $i=n$ can be dropped from the sum in \eqref{eq:rxl-lower-step}, because $\frac{n+1-n}{2} \log \frac{d+1-n}{2} < 0$ only if $d=n$, in which case the negative quantity $\frac 12 \log \frac 12$ is subsumed by the $-O(n)$ term.
Furthermore, since the function $x \mapsto \frac{n+1-x}{2} \log \frac{d+1-x}{2}$ is decreasing on $[1,n]$, we have
\begin{align*}
&\sum_{i=1}^{n-1} \left( \frac{n+1-i}{2} \log \frac{d+1-i}{2} \right) 
\ge \int_1^n \frac{n+1-x}{2} \log \frac{d+1-x}{2} \, dx \\
&= \frac{2dn - d^2}{4} \log \frac{d}{d+1-n} + \frac{n^2-1}{4} \log (d + 1 - n) + O(n^2) ,
\end{align*}
where the integral can be evaluated explicitly but we suppress $O(n^2)$ terms for brevity.
Plugging this back into \eqref{eq:rxl-lower-step}, we obtain
\begin{equation*}
R_X^L(D) 
\ge \frac{n(n+1)}{4} \log \frac{1}{D} + \frac{2dn - d^2 - n^2 + 1}{4} \log \frac{d}{d+1-n} - O(n^2) .
\end{equation*}
Since $2dn - d^2 - n^2 \le 0$ and $\log \frac{d}{d+1-n} \le \frac{n-1}{d+1-n} \le \frac{n-1}{d-n}$, it holds that
$$
\frac{2dn - d^2 - n^2 + 1}{4} \log \frac{d}{d+1-n}
\ge \frac{2dn - d^2 - n^2}{4} \cdot \frac{n-1}{d-n} = -\frac 14 (d-n)(n-1) .
$$
(While the above argument relied on $d>n$ due to the presence of $d-n$ in the denominator, the conclusion clearly holds for $d=n$.)
Consequently, we again have
\begin{equation*}
R_X^L(D) 
\ge \frac{n(n+1)}{4} \log \frac{1}{D} - O(n^2) .
\end{equation*}
This readily implies the desired lower bound.

\subsection{Case $c^* n < d < n$}
\label{sec:middle d} 



This case can be easily reduced to the case $d \ge n$. 
Fix a conditional distribution $P_{Y \mid X}$ such that $\E L(X,Y) \le D$. 
Let $X_d$ be the top left $d \times d$ principal minor of $X$ and define $Y_d$ similarly.
Then $X_d$ clearly has the Wishart distribution as $X$ in Theorem~\ref{thm:wishart-rate-distortion} with $n$ replaced by $d$.
Let $L_d$ be the loss $L$ in \eqref{eq:loss-function-L} with $n$ replaced by $d$. 
Then we have
$$
L_d(X_d, Y_d) = \frac{d}{d(d+1)} \|X_d - Y_d\|_F^2 
\le \frac{d}{(c^*)^2 n(n+1)} \|X - Y\|_F^2 
= \frac{1}{(c^*)^2} L(X,Y) ,
$$
so $\E L_d(X_d,Y_d) \le D/(c^*)^2$.
Applying the result for the case $d = n$, we get
$$
I(X_d; Y_d) \ge \frac{d(d+1)}{4} \log \frac{(c^*)^2}{D} - O(d^2)
\ge \frac{c^* nd}{4} \log \frac{(c^*)^2}{D} - O(nd) .
$$
Since $I(X;Y) \ge I(X_d;Y_d)$, to complete the proof, it remains to take $D \le c$ for a sufficiently small constant $c>0$ depending only on $c^*$ and the hidden constant in $O(nd)$. 

\subsection{Spherical case}
\label{sec:spherical}
We now consider the case $Z = [z_1 \dots z_n]^\top$ and $X = Z Z^\top$ where $z_1, \dots, z_n$ are i.i.d.\ uniform random vectors over the unit sphere $\cS^{d-1} \subset \R^d$. 
The proof is via a reduction from the Gaussian case.
Let $w_1, \dots, w_n$ be i.i.d.\ $\cN(0, \frac 1d I_d)$ vectors and let $\beta_i := \|w_i\|_2$, so that $z_i = w_i/\beta_i$ and $w_i = \beta_i z_i$. 
Let $B \in \R^{n \times n}$ be the diagonal matrix with $\beta_1, \dots, \beta_n$ on its diagonal.
Let $Y = BXB$. Then $Y$ has the distribution of $X$ in the case where $z_1, \dots, z_n$ are Gaussian vectors, so the result of the Gaussian case gives
\begin{equation}
\label{eq:s1-rate-dist}
R_Y^L(D) \ge c n (n \land d) \log \frac{1}{D} .
\end{equation}

Fix a conditional distribution $P_{\hat X \mid X}$ such that $\E L(X, \hat X) \le D$. 
Let $g_1, \dots, g_n$ be i.i.d.\ $\cN(0, \delta^2)$ random variables independent from everything else, where $\delta > 0$ is to be chosen. 
Define $\hat \beta_i := \beta_i + g_i$, and let $\hat B \in \R^{n \times n}$ be the diagonal matrix with $\hat \beta_1, \dots, \hat \beta_n$ on its diagonal. 
Define $\hat Y := \hat B \hat X \hat B$. Since $z_i$ is independent from $\beta_i$, we see that $(X, \hat X)$ is independent from $(B, \hat B)$. 
Hence,
$$
I(Y; \hat Y) \le I(X,B; \hat X, \hat B) = I(X; \hat X) + I(B; \hat B). 
$$
For the term $I(B; \hat B)$, the independence across the pairs $(\beta_i, \hat \beta_i)$ for $i = 1, \dots, n$ implies 
$$
I(B; \hat B) = \sum_{i=1}^n I(\beta_i; \hat \beta_i)
= n I(\beta_1 ; \hat \beta_1). 
$$
We have $\Var(\beta_1) = \Var(\|w_i\|_2) = \frac{1}{d} (d - 2 \frac{\Gamma((d+1)/2)^2}{\Gamma(d/2)^2}) \le 1/(2d)$ using the variance of the $\chi_d$ distribution and basic properties of the gamma function. 
Let $g' \sim \cN(0,1/(2d))$.
Then the Gaussian saddle point theorem (see Theorem~5.11 of \cite{YihongITbook}) gives
$$
I(\beta_1; \hat \beta_1) \le I(g';g'+g_1) 
= \frac 12 \log \left( 1 + \frac{1}{2d\delta^2} \right)
.
$$
The above three displays combined yield
\begin{equation}
\label{eq:s2-mut-inf}
I(X; \hat X) \ge I(Y; \hat Y) - \frac n2 \log \left( 1 + \frac{1}{2d\delta^2} \right).
\end{equation}

It remains to bound $I(Y; \hat Y)$ from below. 
To this end, note that
\begin{align*}
\|\hat Y - Y\|_F^2
&= \|\hat B \hat X \hat B - BXB\|_F^2 \\
&\le 2 \|\hat B \hat X \hat B - \hat B X \hat B\|_F^2 + 2 \|\hat B X \hat B - B X B\|_F^2 \\
&= 2 \sum_{i,j=1}^n \hat \beta_i^2 \hat \beta_j^2 (\hat X_{ij} - X_{ij})^2 
+ 2 \sum_{i,j=1}^n X_{ij}^2 (\hat \beta_i \hat \beta_j - \beta_i \beta_j)^2 .
\end{align*}
Since $\hat \beta_i = \beta_i + g_i$, we have $\E[\hat \beta_i^2] = \E[\beta_i^2] + \E[g_i^2] = 1+\delta^2$.
Moreover, we have $\E[X_{ii}^2] = \E[(z_i^\top z_i)^2] = 1$ and $\E[X_{ij}^2] = \E[(z_i^\top z_j)^2] = 1/d$ for $i \ne j$. 
Finally, 
$$
\E[(\hat \beta_i \hat \beta_j - \beta_i \beta_j)^2] 
= \E[(\beta_i g_j + \beta_j g_i + g_i g_j)^2]
= 2 \delta^2 + \E[g_i^2 g_j^2] + 2 \E[\beta_i \beta_j] \E[g_i g_j] 
$$
so $\E[(\hat \beta_i^2 - \beta_i^2)^2] = 4 \delta^2 + 3 \delta^4$ and $\E[(\hat \beta_i \hat \beta_j - \beta_i \beta_j)^2] = 2 \delta^2 + \delta^4$ for $i \ne j$.
Since $\hat \beta_1, \dots, \hat \beta_n$ are independent and $B, \hat B, X$ are mutually independent, we conclude that
\begin{align*}
\E \|\hat Y - Y\|_F^2
&\le 2 (1+\delta^2)^2 \E \|\hat X - X\|_F^2 
+ 2 n (4 \delta^2 + 3 \delta^4) + 2 \frac{n(n-1)}{d} (2 \delta^2 + \delta^4) \\
&\le 8 \frac{n(n+1)}{d} D 
+ 14 \frac{n}{d} D + 6 \frac{n(n-1)}{d^2} D ,
\end{align*}
where we used that $\E L(X,\hat X) \le D$ for the loss $L$ defined in \eqref{eq:loss-function-L} and chose $\delta^2 = D/d < 1$.
Hence, we have $\E L(Y, \hat Y) \le 28 D$.
This together with \eqref{eq:s1-rate-dist} implies that 
$$
I(Y; \hat Y) \ge c n (n \land d) \log \frac{1}{28D}.
$$
Plugging this bound into \eqref{eq:s2-mut-inf}, we obtain 
$$
I(X; \hat X) \ge c n (n \land d) \log \frac{1}{28D} - \frac n2 \log \left( 1 + \frac{1}{2 D} \right) .
$$

The above bound completes the proof if $d \ge C$ for some constant $C > 0$ depending only on $c$.
For the case $d \le C$ (in fact, for the entire case $d \le c^* n$), it suffices to note that the proof in Section~\ref{sec:small d} also works for the spherical model. 
To be more precise, there are only three places where the Gaussianity assumption is used. 
First, the proof of Lemma~\ref{lem:rate-distortion-comparison} uses the orthogonal invariance of the distribution of the rows of $Z$, which is also true for the spherical model where $z_i$ is uniform over $\cS^{d-1}$. 
Second, \eqref{eq:gaussian-matrix-rate-distortion} uses the rate-distortion function of the entrywise Gaussian matrix $Z$. In the case where $Z$ have i.i.d.\ rows distributed uniformly over $\cS^{d-1}$, it suffices to replace this formula by a lower bound: By Theorems~27.17 and~24.8 of \cite{YihongITbook}, we have
$$
R_Z^{L_0}(D) \ge \frac{n(d-1)}{2} \log \frac{1}{D} - n C_2
$$
for an absolute constant $C_2 > 0$, which is sufficient for the rest of the proof.
Third, the proof of Lemma~\ref{lem:rate-distortion-rotation moduling rotation} also uses that $\E \|Z\|^2$ is of order $\frac{n+d}{d}$, which is obviously true if $d$ is of constant size and the rows of $Z$ are on the unit sphere.

\section*{Acknowledgments}
This work was supported in part by NSF grants DMS-2053333, DMS-2210734, and DMS-2338062. 
We thank Shuangping Li, Eric Ma, and Tselil Schramm for generously sharing their different approach to a similar result on random geometric graphs; the two works were developed concurrently and independently. We thank Yihong Wu and Jiaming Xu for helpful discussions on the rate-distortion theory.

\bibliographystyle{alpha}
\bibliography{IT}

\end{document}